\documentclass[12pt]{article}
\usepackage{amssymb,amsfonts,amsmath,amsthm,amsrefs,hyperref,tikz,wrapfig}
\usepackage{vmargin}
\usepackage{dsfont}
\setpapersize{A4}
\setmargnohfrb{3cm}{3cm}{3cm}{3cm}

\newcommand{\Z}{\mathbb{Z}}

\newcommand{\R}{\mathbb{R}}
\newcommand{\C}{\mathbb{C}}

\newcommand{\F}{\mathbb{F}}

\newcommand{\8}{\infty}
\newcommand{\ph}{\frac{\pi}{2}}

\newcommand{\spa}{\mathrm{span}}

\newcommand{\re}{\mathrm{Re~}}

\newcommand{\diag}{\mathrm{diag}}

\newcommand{\Co}{\mathcal{C}}

\newcommand{\proj}{\mathrm{proj}}

\newcounter{dummy} \numberwithin{dummy}{section}
\newtheorem{theorem}[dummy]{Theorem}
\newtheorem{lemma}[dummy]{Lemma}

\newtheorem{proposition}[dummy]{Proposition}
\newtheorem{corollary}[dummy]{Corollary}

\theoremstyle{remark}
\newtheorem{remark}[dummy]{Remark}
\newtheorem{example}[dummy]{Example}

\begin{document}

\title{Isometry invariant Finsler metrics on Hilbert Spaces}
\author{Eugene Bilokopytov\footnote{Email address ievgen.bilokopytov@umanitoba.ca.}}
\maketitle

\begin{abstract}
In this paper we study isometry-invariant Finsler metrics on inner product spaces over $\R$ or $\C$, i.e. the Finsler metrics which do not change under the action of all isometries of the inner product space. We give a new proof of the analytic description of all such metrics. In this article the most general concept of the Finsler metric is considered without any additional assumptions that are usually built into its definition. However, we present refined versions of the described results for more specific classes of metrics, including the class of Riemannian metrics. Our main result states that for an isometry-invariant Finsler metric the only possible linear maps under which the metric is invariant are scalar multiples of isometries. Furthermore, we characterise the metrics invariant with respect to all linear maps of this type.

Finsler metric, unitary invariance, isometries, Riemannian metric;
MSC2010 53B40, 53C60, 58B20;
\end{abstract}

\section{Introduction}

Geometric objects are largely characterized by their symmetries. An important class of geometric objects is formed by manifolds equipped with Finsler metrics. The subject of this article is the class of such manifolds, which are located inside of a certain inner product space and share symmetries with the later. Namely, we are interested in the Finsler metrics that are invariant with respect to all isometries of the ambient space.

This paper is dedicated to two questions: finding an analytic description of the isometry-invariant Finsler metrics and determining which further symmetries such metrics can have. The definition and the concept of a Finsler metric substantially vary through the literature. In fact, our first question is already answered in the sense of one of these definitions, in \cite{zhou} for the real case, and in \cite{xz} for the complex case\footnote{Note, that a partial case is considered in \cite{mr} from a completely different standpoint.}. However, here we present a simpler proof and a slightly different formulation, and also deal with other definitions of the Finsler metric. Moreover, we are not confined to finite dimensions, as is the case in the cited papers. \\

Let $\F$ be the field of either real or complex numbers. Unless stated otherwise, we will treat these two cases simultaneously, and so, for example, the word "sesquilinear" for the real case means simply "bilinear". By an isometry we mean an $\F$-linear operator of an inner product space that preserves the inner product. This is equivalent to the assumption that the operator preserves the norm induced by the inner product. Note that the alternative name for the isometries of $\R^{n}$ is "\emph{orthogonal operators}", while the \textbf{surjective} isometries of a complex Hilbert space of any dimensions are called \emph{unitary operators}, or \emph{unitaries}. In the light of this fact, we will also call surjective isometries unitaries. Clearly, these terms coincide in finite dimensions.

Since all the manifolds under consideration are domains in an inner product space, for simplicity we will not invoke the language of tangent spaces.

Let $H$ be an inner product space over $\F$ and let $G$ be an open set in $H$. A \emph{non-symmetric Finsler metric} on $G$ is a function $\rho:G\times H\to\R$ such that $\rho_{g}\left(r h\right)=r\rho_{g}\left(h\right)$, for every $g\in G, h\in H, r\ge 0$. The term \emph{Finsler metric} will be reserved for functions $\rho:G\times H\to\R$ such that $\rho_{g}\left(r h\right)=\left|r\right|\rho_{g}\left(h\right)$, for every $g\in G, h\in H, r\in\F$. Note, that we do not require $\rho$ to be non-negative and we do not make any assumptions related to the smoothness or shape of $\rho$, including subadditivity.

A \textbf{non-negative} Finsler metric can be used to define a distance in two stages: first, for a $\Co^{1}$ curve $\gamma: \left[a,b\right]\to G$ define the length by $\int\limits_{a}^{b}\rho_{\gamma\left(t\right)}\left(\gamma'\left(t\right)\right)dt$; then the distance between two points is the infinum of the lengths of the curves that contain these points. For a detailed account on the Finsler metrics see for example \cite{ap}.

The terms "Riemannian metric" and "Hermitean metric" are much more settled and the standard definitions involve smoothness and positive definiteness. Hence, we will not assign any specific name for a scalar function $\sigma$ on $G\times H\times H$ such that $\sigma_{g}$ is a conjugate-symmetric sesquilinear form on $H$, for every $g\in G$. For every such function we can define an associated Finsler metric $\rho^{\sigma}_{g}\left(h\right)=\frac{\sigma_{g}\left(h,h\right)}{\sqrt{\sigma_{g}\left(h,h\right)}}$, and so all concepts related to Finsler metrics can be applied to $\sigma$. In particular, if $\sigma_{g}$ is positive definite for every $g\in G$, then $\rho^{\sigma}_{g}\left(h\right)=\sqrt{\sigma_{g}\left(h,h\right)}$ is a non-negative Finsler metric. \\

We will say that a function $\rho$ on $G\times H$ is \emph{invariant} with respect to a linear map $T:H\to H$ if $TG\subset G$ and $\rho_{Tg}\left(Th\right)=\rho_{g}\left(h\right)$ for any $g\in G, h\in H$. Then $T$ is called a \emph{symmetry} of $\rho$. Clearly, the symmetries of any function form a \emph{monoid}, while the \textbf{invertible} symmetries form a group. We will call $\rho$ \emph{isometry-invariant} (or \emph{isometrically symmetric}) if it is invariant with respect to all isometries. The invariance of functions on $G\times H\times H$ is defined similarly.\\

Let us depict the contents of the article. In Section \ref{1} we establish some descriptions of isometry-invariant Finsler metrics and study some of their properties. We also deal with isometry-invariant Riemannian/Hermitean metrics.

Section \ref{2} is dedicated to studying metrics which are even more symmetric, in the sense that apart of the isometry invariance, they are also invariant with respect to some other operators on $H$. We will see that the range of such possibilities is very limited. Namely, our main result, Theorem \ref{main} states that with one exception the only invariance compatible with the isometry-invariance is the invariance with respect to the \emph{congruencies} (non-zero constants times isometries). Consequently, it is important to study the congruence invariance, which for isometrically symmetric Finsler metrics is reduced to studying the invariance with respect to \emph{homotheties} (multiplications by constants), and we do it in the beginning of the section.

In Section \ref{3} we state the complementary results for the non-symmetric Finsler metric. The majority of the proofs are omitted since they are analogous to the proofs for the symmetric case.

\section{Isometry-invariant Finsler metrics}\label{1}

Let $G$ be an isometry-invariant open set in $H$, i.e. $G=\bigcup\limits_{r\in R}rS$, where $S$ is the unit sphere of $H$ and $R$ is an open set in $\left[0,+\8\right)$. Since a sphere of a zero radius is just a single point, the zero-vector is a "singularity" in our considerations. The following proposition allows us to remove it from the most of the further discussion.

\begin{proposition}\label{zero}
Let $0\in G$ and let $\rho$ be a scalar function on $G\times H$. Then $\rho$ is isometry-invariant if and only if $\rho\left|_{G\backslash\left\{0\right\}\times H}\right.$ is isometry-invariant and there is (a unique) $b\in\R$, such that $\rho_{0}\left(h\right)=b\|h\|$, for every $h\in H$.
\end{proposition}
\begin{proof}
The sufficiency is obvious; let us prove the necessity. Fix a unit vector $e\in H$. For any $h\in H$ there is an isometry $T$ such that $Th=\|h\|e$. Then  $$\rho_{0}\left(h\right)=\rho_{T0}\left(Th\right)=\rho_{0}\left(\|h\|e\right)=\|h\|\rho_{0}\left(e\right).$$ Thus, for $b=\rho_{0}\left(e\right)$ we obtain that $\rho_{0}\left(h\right)=b\|h\|$, for every $h\in H$.
\end{proof}

Our further discussion is based on the following lemma.

\begin{lemma}\label{ii}
Let $\alpha\in\R$, let $0\not\in G$ and let $\rho:G\times H\to\R$ be such that\\ $\rho_{g}\left(r h\right)=\left|r\right|^{\alpha}\rho_{g}\left(h\right)$, for every $g\in G, h\in H, r\in\F$. Then $\rho$ is isometry-invariant if and only if there is a (unique) function $\lambda:R\times\R^{2}\to\R$, such that:
\begin{itemize}
\item $\lambda_{r}\left(tp, tq\right)=t^{\alpha}\lambda_{r}\left(p,q\right)$ and $\lambda_{r}\left(-p, q\right)=\lambda_{r}\left(p, -q\right)=\lambda_{r}\left(p, q\right)$, for any\\ $r\in R, p,q\in\R, t\ge 0$;
\item $\rho_{g}\left(h\right)=\lambda_{\|g\|}\left(\left|\left<h,g\right>\right|,\sqrt{\|h\|^{2}\|g\|^{2}-\left|\left<h,g\right>\right|^{2}}\right)$, for every $g\in G, h\in H$.
\end{itemize}
\end{lemma}
\begin{proof}
Sufficiency is obvious; let us show necessity. It is easy to see that there is nothing to prove in case when $\dim H\le 1$, so assume that $\dim H> 1$. Fix two orthogonal unit vectors $e,f\in H$. Then, for any $g\in G$ and $h\in H$ there is an isometry $T$ such that $Tg=\|g\|e$ and $$T\proj_{g^{\bot}}h=\frac{\left<h,g\right>}{\left|\left<h,g\right>\right|}\|\proj_{g^{\bot}}h\|f=\frac{\left<h,g\right>}{\left|\left<h,g\right>\right|}\sqrt{\|h\|^{2}-\frac{\left|\left<h,g\right>\right|^{2}}{\|g\|^{2}}}f$$ (if $g\bot h$, we take $\frac{\left<h,g\right>}{\left|\left<h,g\right>\right|}=1$). Then $T\proj_{g} h=T\frac{\left<h,g\right>}{\|g\|^{2}}g=\frac{\left<h,g\right>}{\|g\|^{2}}Tg=\frac{\left<h,g\right>}{\|g\|}e$, and so $$Th=\frac{\left<h,g\right>}{\|g\|}e+\frac{\left<h,g\right>}{\left|\left<h,g\right>\right|}\sqrt{\|h\|^{2}-\frac{\left|\left<h,g\right>\right|^{2}}{\|g\|^{2}}}f=
\frac{\left<h,g\right>}{\left|\left<h,g\right>\right|}\frac{\left|\left<h,g\right>\right|e+\sqrt{\|h\|^{2}\|g\|^{2}-\left|\left<h,g\right>\right|^{2}}f}{\|g\|}.
$$\begin{eqnarray*}
\rho_{g}\left(h\right)&=&\rho_{Tg}\left(Th\right)=\rho_{\|g\|e}\left(\frac{\left<h,g\right>}{\left|\left<h,g\right>\right|}\frac{1}{\|g\|}\left(\left|\left<h,g\right>\right|e+\sqrt{\|h\|^{2}\|g\|^{2}-\left|\left<h,g\right>\right|^{2}}f\right)\right)=\\
&=&\frac{1}{\|g\|^{\alpha}}\rho_{\|g\|e}\left(\left|\left<h,g\right>\right|e+\sqrt{\|h\|^{2}\|g\|^{2}-\left|\left<h,g\right>\right|^{2}}f\right).
\end{eqnarray*}

Thus for $\lambda_{r}\left(p,q\right)=\frac{1}{r^{\alpha}}\rho_{re}\left(pe+qf\right)$ we obtain that $$\rho_{g}\left(h\right)=\lambda_{\|g\|}\left(\left|\left<h,g\right>\right|,\sqrt{\|h\|^{2}\|g\|^{2}-\left|\left<h,g\right>\right|^{2}}\right),$$ for every $g\in G\backslash\left\{0\right\}, h\in H$. Additionally for every $r\in R, p,q,t\in\R$ we have that
$$\lambda_{r}\left(tp, tq\right)=\frac{1}{r^{\alpha}}\rho_{re}\left(t\left(pe+qf\right)\right)=\frac{1}{r^{\alpha}}\left|t\right|^{\alpha}\rho_{re}\left(pe+qf\right)=\left|t\right|^{\alpha}\lambda_{r}\left(p,q\right).$$

Since $e\bot f$, we can find an isometry $T$ such that $Te=e$ and $Tf=-f$. Then $$\lambda_{r}\left(p,-q\right)=\frac{1}{r^{\alpha}}\rho_{re}\left(pe-qf\right)=\frac{1}{r^{\alpha}}\rho_{rTe}\left(pTe-qTf\right)=\frac{1}{r^{\alpha}}\rho_{re}\left(pe+qf\right)=\lambda_{r}\left(p,q\right);$$ finally,  $\lambda_{r}\left(-p,q\right)=\lambda_{r}\left(\left(-1\right)\left(p,-q\right)\right)=1^{\alpha}\lambda_{r}\left(p,-q\right)=\lambda_{r}\left(p,q\right)$.
\end{proof}

As a particular case we obtain an analytic description of the isometry-invariant Finsler metrics.

\begin{proposition}\label{ui}
Let $\rho$ be a Finsler metric on $G\not\ni 0$. Then $\rho$ is isometry-invariant if and only if there is a (unique) function $\lambda:R\times\R^{2}\to\R$, such that:
\begin{itemize}
\item $\lambda_{r}\left(tp, tq\right)=t\lambda_{r}\left(p,q\right)$ and $\lambda_{r}\left(-p, q\right)=\lambda_{r}\left(p, -q\right)=\lambda_{r}\left(p, q\right)$, for any $r\in R, p,q\in\R, t\ge 0$;
\item $\rho_{g}\left(h\right)=\lambda_{\|g\|}\left(\left|\left<h,g\right>\right|,\sqrt{\|h\|^{2}\|g\|^{2}-\left|\left<h,g\right>\right|^{2}}\right)$, for every $g\in G, h\in H$.
\end{itemize}
\end{proposition}

From the above proposition it is clear that $\rho_{g}\left(h\right)$ depends on the "angle" between $g$ and $h$. Let us formalize this idea. For non-zero vectors $g,h$ define the \emph{acute angle} between them, i.e. the minimal angle between the \emph{real rays} lying on the $\F$-lines, containing them by $$\angle \left(g,h\right)=\angle \left(h,g\right)=\cos^{-1}\frac{\left|\left<h,g\right>\right|}{\|h\|\|g\|}.$$ Defining $\theta\left(r,\tau\right)=r\lambda_{r}\left(\cos \tau,\sin \tau\right)$ on $R\times\left[0,\ph\right]$ we obtain another analytic description of isometrically-symmetric Finsler metric.

\begin{corollary}\label{theta}
A Finsler metric $\rho$ on $G\not\ni 0$ is isometry-invariant if and only if there is a (unique) function $\theta:R\times\left[0,\ph\right]\to\R$ such that  $\rho_{g}\left(h\right)=\|h\|\theta\left(\|g\|,\angle \left(g,h\right)\right)$, for every $g\in G, h\in H\backslash \left\{0\right\}$.
\end{corollary}

\begin{remark}\label{dim}
It can be seen from the proof above, that it is enough to demand invariance only with respect to the group of \textbf{unitaries}. Moreover, if $\8>\dim H>2$ it is enough to demand only invariance with respect to the group \emph{rotations} (isometries of determinant $1$). In the light of this fact, the terms unitary- and isometry- invariant Finsler metrics are interchangeable.
\end{remark}

\begin{remark}\label{pr}
It is easy to see that $\lambda_{r}$ has a certain property if and only if $\rho_{g}$ has this property for every $g\in rS$ and if and only if $\rho_{re}$ has this property on some two-dimensional subspace containing some unit vector $e$. Examples of such properties include non-negativity, continuity, being a (semi)norm, smoothness etc. In particular, if we assume that $\rho_{g}$ is a seminorm for \textbf{some} $g\in rS$, we get that $\rho_{g}$ is a \textbf{continuous} seminorm for \textbf{all} $g\in rS$. The global properties of $\rho$ and $\lambda$ are also connected. For example, $\rho$ is continuous if and only if $\lambda$ is.
\end{remark}

\begin{remark}\label{dg}
If $\lambda_{r}$ it is a seminorm, then it is either a norm or $\lambda_{r}\left(p,q\right) = a\left|p\right|$, for some $a\ge 0$, or $\lambda_{r}\left(p,q\right) = b\left|q\right|$, for some $b\ge 0$. Indeed, the null-space of a seminorm is a subspace of $\R^{2}$, which in our case has to be symmetric with respect to both axis. Thus, it is either one of them, or the whole $\R^{2}$, or $\left\{0\right\}$. Note that if $\lambda_{r}\left(p,q\right) = a\left|p\right|$, then the length of any curve inside of $rS$ is $0$, and so $\rho$ glues all elements with norm equal to $r$. See Remark \ref{dg2} for further details.
\end{remark}

Let us deal with the isometry-invariant Riemannian/Hermitean metrics.

\begin{proposition}\label{rm1}
Let $0\not\in G$ and let $\sigma:G\times H\times H\to\F$ be such that $\sigma_{g}$ is conjugate-symmetric sesquilinear on $H$, for every $g\in G$. Then $\sigma$ is isometry-invariant if and only if there are (unique) functions $\varphi,\psi:R\to\R$, such that $\sigma_{g}\left(f,h\right)=\varphi\left(\|g\|^{2}\right)\left<f,h\right>+\psi\left(\|g\|^{2}\right)\left<f,g\right>\left<g,h\right>$, for $g\in G$, and $f,h\in H$. Moreover, in this case the following hold:
\item[(i)] $\sigma_{g}$ is positive definite for some (every) $g\in \sqrt{r}S$ if and only if $\varphi\left(r\right)> 0$ and $\varphi\left(r\right)+r\psi\left(r\right)> 0$;
\item[(ii)] the degree of smoothness of $\sigma$ coincides with the minimal degree of smoothness of $\varphi$ and $\psi$.
\end{proposition}
\begin{proof}
Sufficiency is obvious; let us show necessity. It is easy to see that $\rho_{g}\left(h\right)=\sigma_{g}\left(h,h\right)$ satisfies the conditions of Lemma \ref{ii} with $\alpha=2$, and so there is a (unique) function $\lambda:R\times\R^{2}\to\R$, such that $\lambda_{r}\left(-p, q\right)=\lambda_{r}\left(p, -q\right)=\lambda_{r}\left(p, q\right)$ and $\lambda_{r}\left(tp, tq\right)=t^{2}\lambda_{r}\left(p,q\right)$, for any $r\in R, p,q\in\R, t\ge 0$, and also $$\rho_{g}\left(h\right)=\lambda_{\|g\|}\left(\left|\left<h,g\right>\right|,\sqrt{\|h\|^{2}\|g\|^{2}-\left|\left<h,g\right>\right|^{2}}\right),$$ for every $g\in G, h\in H$.

We again may assume that $\dim H> 1$ and fix two orthogonal unit vectors $e,f\in H$. Then $\lambda_{r}\left(p,q\right)=\frac{1}{r^{2}}\rho_{re}\left(pe+qf\right)=\sigma_{re}\left(pe+qf,pe+qf\right)$, and so $\lambda_{r}$ is a quadratic form on $\R^{2}$. Since $\lambda_{r}$ is also even in both of its variables, it is easy to show that there are real numbers $\upsilon\left(r\right)$ and $\xi\left(r\right)$, such that $\lambda_{r}\left(p,q\right)=\upsilon\left(r\right)\left|p\right|^{2}+\xi\left(r\right)\left|q\right|^{2}$. Define $\varphi\left(r\right)=r\xi\left(\sqrt{r}\right)$ and $\psi\left(r\right)=\upsilon\left(\sqrt{r}\right)-\xi\left(\sqrt{r}\right)$, for $r\in R$. For $g\in G$ and $h\in H$ we have that
\begin{eqnarray*}
\rho_{g}\left(h\right)&=&\lambda_{\|g\|}\left(\left|\left<h,g\right>\right|,\sqrt{\|h\|^{2}\|g\|^{2}-\left|\left<h,g\right>\right|}\right)^{2}\\
&=&\upsilon\left(\|g\|\right)\left|\left<h,g\right>\right|^{2}+\xi\left(\|g\|\right)\left(\|h\|^{2}\|g\|^{2}-\left|\left<h,g\right>\right|^{2}\right)=\varphi\left(\|g\|^{2}\right)\|h\|^{2}+\psi\left(\|g\|^{2}\right)\left|\left<f,g\right>\right|^{2}
\end{eqnarray*}
Since $\rho_{g}$ is a quadratic form, it uniquely determines the corresponding conjugate-symmetric sesquilinear form. Thus the main statement follows.\\

(i): If $\sigma_{g}$ is positive definite, then $0<\sigma_{g}\left(f,f\right)=\varphi\left(\|g\|^{2}\right)\|f\|^{2}+\psi\left(\|g\|^{2}\right)\left|\left<f,g\right>\right|^{2}$, for any $f,g$. Hence,
$0< \varphi\left(\|g\|^{2}\right)+\psi\left(\|g\|^{2}\right)\left|\left<f,g\right>\right|^{2}$, for any $g\in G$ and any unit vector $f$. The quantity $\left|\left<f,g\right>\right|^{2}$ can have any value from $0$ to $\|g\|^{2}$, and so our condition is reduced to $0< \varphi\left(\|g\|^{2}\right)$ and $0< \varphi\left(\|g\|^{2}\right)+\psi\left(\|g\|^{2}\right)\|g\|^{2}$, or $\varphi\left(r\right)> 0$ and $\varphi\left(r\right)+r\psi\left(r\right)> 0$. Reversing the implications we get that these conditions are also sufficient for positive definiteness of $\sigma_{g}$.

(ii): If $\varphi$ and $\psi$ satisfy a certain smoothness condition, so does $\sigma$, since it is expressed through $\varphi$ and $\psi$. Conversely, if $\sigma$ is smooth, by taking $f\bot h$ and letting $g=\sqrt{r}e$ for a unit $e$ not orthogonal to $f,h$ we get that\\
$r\psi\left(r\right)\left<f,e\right>\left<e,h\right>=\sigma_{\sqrt{r}e}\left(f,h\right)$, and so $\psi$ is smooth, while for $f=h=e$ and $g=\sqrt{r}e$, we get that $\varphi\left(r\right)=\sigma_{\sqrt{r}e}\left(e,e\right)-r\psi\left(r\right)$ is also smooth.
\end{proof}

\begin{remark}
We can partially extend this proposition to the case when $0\in G$. From Proposition \ref{zero} there is $b\in\R$, such that $\rho_{0}\left(h\right)=b\|h\|$, for every $h\in H$. Hence, for $\varphi\left(0\right)=b$ and \textbf{any} value of $\psi\left(0\right)$ the statement follows. Obviously, the positive definiteness of $\sigma_{0}$ is equivalent to $b>0$. Note that it is not clear how to extend the part (ii).
\end{remark}

\begin{remark}\label{dg2}
The non-strict analogues of the strict inequalities in part (i) correspond to the positive \textbf{semi-definiteness} of $\sigma_{g}$. In particular, if $\varphi\left(r\right)=-r\psi\left(r\right)$, for every $r\in R$, then $\sigma$ glues elements that are scalar multiples of each other, i.e. factorizes by the $\F$-lines. The case when $\varphi\left(r\right)=0$ leads to the situation described in Remark \ref{dg}.
\end{remark}

\begin{remark}\label{k}
If $\F=\C$, then $\sigma$ is a Kaehler metric if and only if $\psi=\varphi'$. In this case $\omega\circ\|\cdot\|^2$ is the potential of this metric, where $\omega'=\varphi$. See \cite{xz} and \cite{zhong} for further details.
\end{remark}

\section{Congruency-invariant Finsler metrics}\label{2}

We start with the case when an isometry-invariant Finsler metric is also invariant with respect to a homotethy with a coefficient $\alpha\in\left(0,1\right)\bigcup\left(1,+\8\right)$. Observe that in this case $\alpha R\subset R$, and if $0\in R$, then $\rho_{0}\equiv 0$. Using Corollary \ref{theta} and defining $\vartheta\left(r,\tau\right)=r\theta\left(r,\tau\right)$ we obtain the following characterization.

\begin{proposition}\label{h}
An isometry-invariant Finsler metric $\rho$ on $H\backslash\left\{0\right\}$ is invariant with respect to a homotethy with a coefficient $\alpha$ if and only if there is a (unique)  function $\vartheta:R\times\left[0,\ph\right]\to\R$, such that $\vartheta\left(\exp\left(\cdot\right),\cdot\right)$ is periodic in the first variable with the period $\ln\alpha$ and $\rho_{g}\left(h\right)=\frac{\|h\|}{\|g\|}\vartheta\left(\|g\|,\angle \left(g,h\right)\right)$, for $g\in G, h\in H\backslash \left\{0\right\}$.
\end{proposition}

Let us deal with the situation when an isometry-invariant Finsler metric is invariant with respect to all homotheties, not just one. Clearly, then $G$ is either $H$ or $H\backslash\left\{0\right\}$ and if $G=H$, then $\rho_{0}\equiv 0$. Furthermore, the following characterization holds.

\begin{proposition} Let $\rho$ be a Finsler metric on $H\backslash\left\{0\right\}$. The following are equivalent:
\item[(i)] $\rho$ is invariant with respect to all congruencies;
\item[(ii)] there is a (unique) function $\vartheta:\left[0,\ph\right]\to\R$ such that $\rho_{g}\left(h\right)=\frac{\|h\|}{\|g\|}\vartheta\left(\angle \left(g,h\right)\right)$, for every $g, h\in H\backslash \left\{0\right\}$;
\item[(iii)] there is a (unique) positive-homogenous function $\lambda:\R\times\R\to\R$, which is even in both variables such that  $\rho_{g}\left(h\right)=\frac{1}{\|g\|^{2}}\lambda\left(\left|\left<h,g\right>\right|,\sqrt{\|h\|^{2}\|g\|^{2}-\left|\left<h,g\right>\right|^{2}}\right)$, for every $g, h\in H\backslash \left\{0\right\}$.

Moreover, if we additionally assume that $\rho$ is continuous in the first variable, then the above conditions are equivalent to
\item[(iv)] $\rho$ is invariant with respect to all isometries and two homotheties such that the logarithms of their coefficients are not commensurable.
\end{proposition}
\begin{proof}
First, note that (ii)$\Rightarrow$(i) is obvious, while (i)$\Leftrightarrow$(ii) follows from putting\\ $\vartheta\left(\tau\right)=\lambda\left(\cos \tau,\sin \tau\right)$. Now assume that $\rho$ is an isometry-invariant Finsler metric on $H\backslash\left\{0\right\}$. From the discussion before the previous proposition, there is a (unique) function $\vartheta:R\times\left[0,\ph\right]\to\R$, such that $\rho_{g}\left(h\right)=\frac{\|h\|}{\|g\|}\vartheta\left(\|g\|,\angle \left(g,h\right)\right)$, for $g, h\in H\backslash \left\{0\right\}$.

(i)$\Rightarrow$(ii): From the previous proposition we get that the function $\vartheta\left(\exp\left(\cdot\right),\cdot\right)$ is periodic in the first variable with the period equal to every non-zero real number. Thus, $\vartheta$ does not depend on the first variable, i.e. $\vartheta=\vartheta\left(\cdot\right)$, and $\rho_{g}\left(h\right)=\frac{\|h\|}{\|g\|}\vartheta\left(\angle \left(g,h\right)\right)$, for every $g, h\in H\backslash \left\{0\right\}$.

(iv)$\Rightarrow$(ii): It is easy to see that the continuity of $\rho$ in the first variable implies the continuity of $\vartheta$ in the first variable. Assume that $\rho$ is invariant with respect to homotethety with coefficients $\alpha$ and $\beta$. From the previous proposition we get that the function $\vartheta\left(\exp\left(\cdot\right),\cdot\right)$ is periodic in the first variable with the periods $\ln \alpha$ and $\ln \beta$. Then $n\ln \alpha+m\ln \beta$ are also periods, for any $m,n\in\Z$. If $\ln \alpha$ and $\ln \beta$ are not commensurable, by Kroneker's theorem the latter numbers densely fill the real line, and by continuity we obtain that $\vartheta$ is a constant with respect to the first variable.
\end{proof}

\begin{remark}\label{pr2}
It is easy to come up with the refinements of remarks \ref{pr} and \ref{dg} for this case. Also, it follows that $\rho$ blows up to infinity as $g$ approaches $0$, unless $\rho\equiv 0$.
\end{remark}

\begin{corollary}
Let $G=H\backslash \left\{0\right\}$ and let $\sigma$ be as in Proposition \ref{rm1}. Then $\sigma$ is invariant with respect to all congruencies if and only if there are (unique) $a,b\in\R$, such that $\sigma_{g}\left(f,h\right)=\frac{a}{\|g\|^{2}}\left<f,h\right>+\frac{b}{\|g\|^{4}}\left<f,g\right>\left<g,h\right>$, for $g\in G$, and $f,h\in H$.
\end{corollary}

\begin{remark}
Combining the previous corollary with Proposition \ref{rm1} and Remark \ref{k} we get the following. The positive definiteness of $\sigma_{\cdot}$ is equivalent to $a+b>0$. If $\F=\C$, the latter contradicts to the necessary condition for $\sigma$ to be Kaehler, which is reduced to $a=-b$. Thus, there is no Kaehler metrics on $H\backslash\left\{0\right\}$ invariant with respect to all congruencies.
\end{remark}

\begin{example}\label{fs}
Let $\sigma_{g}\left(f,h\right)=\frac{1}{\|g\|^{2}}\left<f,h\right>-\frac{1}{\|g\|^{4}}\left<f,g\right>\left<g,h\right>$. By the above remark, this is the unique (up to scalar multiplication) "degenerate Kaehler metrics" which is invariant with respect to all congruencies. Using Remark \ref{dg2} one can show that it is also the unique (up to scalar multiplication) "degenerate Kaehler metrics" which factorizes by the $\F$-lines. Since $\sigma$ is the pull-back of the classical Fubiny-Study metric on the projective space $PH$ via the natural quotient map, we find it natural to call $\sigma$ the Fubini-Study metric on  $H\backslash\left\{0\right\}$. Note, that $2\log\|\cdot\|$ is the potential of this metric. Inspired by \cite{arsw} and following \cite{kob}, we will define two congruency-invariant pseudodistances on $H\backslash\left\{0\right\}$ and show that $\sigma$ is the "intrinsification" of them. For $g,h\in H\backslash\left\{0\right\}$ define $$\delta_{1}\left(g,h\right)=\sin\angle\left(g,h\right)=\sqrt{1-\frac{\left|\left<g,h\right>\right|^{2}}{\|g\|^{2}\|h\|^{2}}},~\delta_{2}\left(g,h\right)=\sin\frac{\angle\left(g,h\right)}{2}=\sqrt{2-2\frac{\left|\left<g,h\right>\right|}{\|g\|\|h\|}}.$$ While the geometric meaning of $\delta_{1}$ is obvious, $\delta_{2}$ is the distance between the intersections of the $\F$-lines defined by $g,h$ lines and the unit sphere. For the real case it follows from the low of sines, or from $\delta_{2}\left(g,h\right)=\|\frac{g}{\|g\|}\pm\frac{h}{\|h\|}\|$, where the sign depends on the acuteness of the angle between $g$ and $h$. For the complex case it follows from  $$\inf\limits_{t,s\in\R}\left\|\frac{e^{it}g}{\|g\|}-\frac{e^{is}h}{\|h\|}\right\|=\inf\limits_{t,s\in\R}\sqrt{2-2\re e^{i\left(t-s\right)}\frac{\left<g,h\right>}{\|g\|\|h\|}}=\delta_{2}\left(g,h\right).$$

For the real case the intrinsification of $\delta_{2}$ locally is the arc length of the projection on the unit sphere. Since $\delta_{2}\le\delta_{1}$ and $\delta_{1}$ does not exceed the normalized arc length, the intrinsification of $\delta_{1}$ locally is also the normalized arc length. Now, if $\gamma:\left[a,b\right]\to H\backslash\left\{0\right\}$ is a smooth curve, then
 \begin{eqnarray*}
 \left\|\left(\frac{\gamma}{\|\gamma\|}\right)'\right\|^{2}=\left\|\frac{\|\gamma\|\gamma'-\|\gamma\|'\gamma}{\|\gamma\|^2}\right\|^{2}=\left\|\frac{\|\gamma\|\gamma'-\sqrt{\left<\gamma,\gamma\right>}'\gamma}{\|\gamma\|^2}\right\|^{2}=\frac{1}{{\|\gamma\|^4}}\left\|\|\gamma\|\gamma'-\frac{\left<\gamma,\gamma\right>'}{2\sqrt{\left<\gamma,\gamma\right>}}\gamma\right\|^{2}\\
 =\frac{\left\|\|\gamma\|^{2}\gamma'-\left<\gamma,\gamma'\right>\gamma\right\|^{2}}{{\|\gamma\|^{6}}}=\frac{\left<\|\gamma\|^{2}\gamma'-\left<\gamma,\gamma'\right>\gamma,\|\gamma\|^{2}\gamma'-\left<\gamma,\gamma'\right>\gamma\right>}{{\|\gamma\|^{6}}}\\
 =\frac{\|\gamma\|^{4}\|\gamma'\|^{2}+\left|\left<\gamma,\gamma'\right>\right|^{2}\|\gamma\|^{2}-2\left|\left<\gamma,\gamma'\right>\right|^{2}\|\gamma\|^{2}}{{\|\gamma\|^{6}}}=\frac{\|\gamma\|^{2}\|\gamma'\|^{2}-\left|\left<\gamma,\gamma'\right>\right|^{2}}{\|\gamma\|^{4}}=\sigma_{\gamma}\left(\gamma',\gamma'\right),
 \end{eqnarray*}
and so the normalized length of $\gamma$ is $\int\limits_{a}^{b}\left\|\left(\frac{\gamma\left(t\right)}{\|\gamma\left(t\right)\|}\right)'\right\|dt=\int\limits_{a}^{b}\sqrt{\sigma_{\gamma\left(t\right)}\left(\gamma'\left(t\right),\gamma'\left(t\right)\right)}dt$.
For the complex case, using the first-order Tailor expansion for a $\Co^{1}$ curve $\gamma$, one can prove that $$\lim\limits_{s\to t}\frac{\delta_{1}\left(\gamma\left(s\right),\gamma\left(t\right)\right)}{\left|t-s\right|}=\lim\limits_{s\to t}\frac{\delta_{2}\left(\gamma\left(s\right),\gamma\left(t\right)\right)}{\left|t-s\right|}=\sqrt{\sigma_{\gamma\left(s\right)}\left(\gamma'\left(t\right),\gamma'\left(t\right)\right)},$$ and so by \cite[2.7.3]{burago}, we again arrive at the conclusion that the length of the curves with respect to $\delta_{1}$, $\delta_{2}$ and $\sigma$ coincide. Note that the proofs for the complex case also apply to the real cases.
\end{example}

The following theorem explains the importance of studying the congruency-invariant Finsler metrics.

\begin{theorem}\label{main}
If $\dim H>2$, then any symmetry of any non-zero isometry-invariant Finsler metric is a congruence.
\end{theorem}

Note, that this theorem is a trivial consequence of Theorem \ref{nain}, which is a version of the present theorem for the case of a non-symmetric metric. The latter theorem is again a simple corollary of a general algebraic fact about linear groups. However, we would like to present a more hands-on proof of the present theorem, which transparently reveals its geometric meaning.\\

\emph{Proof.} \textbf{Step 0.} Let $T$ be a symmetry of a non-zero isometry-invariant $\rho$. First of all, let us ascertain that $T$ is an injection. Indeed, if $Te=0$ for some unit vector $e$, then $\rho_{g}\left(h\right)=0$, for any $h\in H$, because there is an isometry $S$, such that $Sh=\|h\|e$, and so for any $g\in G$ we have that  $$\rho_{g}\left(h\right)=\rho_{Sg}\left(Sh\right)=\|h\|\rho_{Sg}\left(e\right)=\|h\|\rho_{TSg}\left(Te\right)=\|h\|\rho_{TSg}\left(0\right)=0.$$

\textbf{Step 1.} Using Corollary \ref{theta} and defining $\vartheta\left(r,\tau\right)=\frac{1}{r}\theta\left(r,\tau\right)$ for $g\in G, h\in H\backslash \left\{0\right\}$ we get that  $\rho_{g}\left(h\right)=\|h\|\|g\|\vartheta\left(\|g\|,\angle\left(g,h\right)\right)$ . If $\|Tg\|=\|Th\|$, for some $g,h\in G\backslash\left\{0\right\}$, then
\begin{eqnarray*} \|g\|\|h\|\vartheta\left(\|g\|,\angle\left(g,h\right)\right)=\rho_{g}\left(h\right)=\rho_{Tg}\left(Th\right)=\|Tg\|\|Th\|\vartheta\left(\|Tg\|,\angle\left(Tg,Th\right)\right)\\
=\|Tg\|\|Th\|\vartheta\left(\|Th\|,\angle\left(Th,Tg\right)\right)=\rho_{Th}\left(Tg\right)=\rho_{h}\left(g\right)=\|g\|\|h\|\vartheta\left(\|h\|,\angle\left(h,g\right)\right)
\end{eqnarray*}
Hence, $\vartheta\left(\|g\|,\angle\left(g,h\right)\right)=\vartheta\left(\|h\|,\angle\left(h,g\right)\right)$, as long as $\|Tg\|=\|Th\|$.

\textbf{Step 2.} Assume that $T$ is not a congruence. Since $T$ is also an injection, there is a two dimensional subspace $E$ of $H$, such that $T$ is an injection and not a constant times an isometry from $E$ onto $TE$. Due to this fact and singular decomposition there are orthogonal vectors $e,f\in E$, such that $Te\bot Tf$, $\|Te\|=\|Tf\|=1$, but $\|e\|\ne\|f\|$. If $F=\spa_{\R}\left\{e,f\right\}$, then $TF=\spa_{\R}\left\{Te,Tf\right\}$ and the ellipse $\Delta$ with axis $e$ and $f$ is mapped into the unit circle of $TF$. Note that the inner product is real on both $F$ and $TF$, and so on these subspaces the acute angle between the vectors according to our definition coincides with the actual acute angle between them on the \textbf{real} planes that contain them.~~~~~~~~~~~~~~~~~~~~~~~~~~~~~~~~~~~~~~~~~~
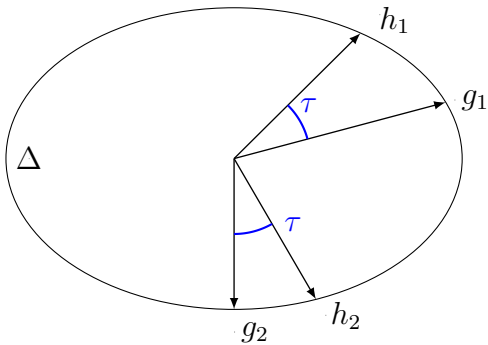
\begin{wrapfigure}{l}{0pt}
\begin{tikzpicture}
\draw[-latex] (0, 0) ellipse [x radius = 3cm, y radius = 2cm,
  start angle = 30, end angle = 150];

\draw[line width=0.5pt, color=black, draw opacity=1, -latex] (0:0)--(15:2.881cm);

\draw[line width=0.5pt, color=black, draw opacity=1, -latex] (0:0)--(45:2.353cm);

\draw[line width=0.5pt, color=black, draw opacity=1, -latex] (0:0)--(-60:2.155cm);

\draw[line width=0.5pt, color=black, draw opacity=1, -latex] (0:0)--(-90:2cm);

\draw[fill] (15:3cm) circle (0pt) node {$~~~~g_{1}$};

\draw[fill] (45:2.7cm) circle (0pt) (45:2.6cm) node {$~~~~h_{1}$};

\draw[fill] (-60:2.4cm) circle (0pt) node {$~~~~h_{2}$};

\draw[fill] (-90:2.3cm) circle (0pt) node {$~~~~g_{2}$};

\draw[fill] (-180:2.7cm) circle (0pt) node {$\Delta$};

\draw [blue,thick,domain=15:45] plot ({cos(\x)}, {sin(\x)}) node {$~~~~\tau$};

\draw [blue,thick,domain=-90:-60] plot ({cos(\x)}, {sin(\x)}) node {$~~~~\tau$};

\end{tikzpicture}
\end{wrapfigure}

\noindent Fix $\tau\in\left(0,\ph\right]$. Consider all pairs of vectors $g$ and $h$ on $\Delta$ such that $\angle\left(g,h\right)=\tau$. It is clear (see the picture), that the ratio $\frac{\|g\|}{\|h\|}$ fills a certain closed interval $\left[c,d\right]$ with $c<1<d$. Since we can expand and shrink $\Delta$ arbitrarily, it follows that $\vartheta\left(s,\tau\right)=\vartheta\left(r,\tau\right)$ for any positive $s,r$ with $\frac{s}{r}\in\left[c,d\right]$. Hence, $\vartheta\left(\cdot,\tau\right)$ is locally a constant on a connected domain $\left(0,+\8\right)$. Thus, $\vartheta\left(\cdot,\tau\right)$ does not depend on the first variable, for any $\tau\in\left(0,\ph\right]$. Using the same letter $\vartheta$ for the function of one (second) variable, we get that $\rho_{g}\left(h\right)=\|g\|\|h\|\vartheta\left(\angle\left(g,h\right)\right)$, unless $g$ and $h$ are collinear.

\textbf{Step 3.} We have that $\|Tg\|\|Th\|\vartheta\left(\angle\left(Tg,Th\right)\right)=\|g\|\|h\|\vartheta\left(\angle\left(g,h\right)\right)$, for any non-collinear $g,h$. Let $S\left(g,h\right)=\sin\left(\angle\left(g,h\right)\right)\|g\|\|h\|$. This quantity is the area of the parallelogram spanned by $g,h$ in the case when $\left<g,h\right>\in\R$. The matrix $\diag\left\{\frac{1}{\|e\|}, \frac{1}{\|f\|}\right\}$ is the matrix of both $T\left|_{E}\right.$ and $T\left|_{F}\right.$ with respect to the orthobases $\frac{e}{\|e\|},\frac{f}{\|f\|}$ and $Te,Tf$. Let $D=\frac{1}{\|e\|\|f\|}>0$ be the determinant of this matrix. Then for $g,h\in\Delta$ we have that
$$\frac{S\left(g,h\right)\vartheta\left(\angle\left(g,h\right)\right)}{\sin \left(\angle\left(g,h\right)\right)}=\frac{S\left(Tg,Th\right)\vartheta\left(\angle\left(Tg,Th\right)\right)}{\sin \left(\angle\left(Tg,Th\right)\right)}=D\frac{S\left(g,h\right)\vartheta\left(\angle\left(Tg,Th\right)\right)}{\sin \left(\angle\left(Tg,Th\right)\right)},$$
and so $\frac{\vartheta\left(\angle\left(g,h\right)\right)}{\sin \left(\angle\left(g,h\right)\right)}=D\frac{\vartheta\left(\angle\left(Tg,Th\right)\right)}{\sin \left(\angle\left(Tg,Th\right)\right)}$. Again, for any fixed $\tau\in \left(0,\ph\right]$ there is an interval of angles $\omega$, such that there are $g,h\in\Delta$ with $\angle\left(g,h\right)=\omega$ and $\angle\left(Tg,Th\right)=\tau$, and consequently $\frac{\vartheta\left(\omega\right)}{\sin \left(\omega\right)}=D\frac{\vartheta\left(\tau\right)}{\sin \left(\tau\right)}$. Hence, $\frac{\vartheta}{\sin}$ is locally a constant on the connected domain $\left(0,\ph\right]$. Thus, $\frac{\vartheta}{\sin}$ is a constant on $\left(0,\ph\right]$, say $b$, and $D=1$.

\textbf{Step 4.} We have shown that the absolute value of the determinant of each restriction of $T$ on a two dimensional subspace with respect to corresponding orthobases is $1$: if this restriction is an isometry, it is automatic, otherwise the argument above applies. Consider a subspace $F$ of $H$ such that $\dim F=3$, and such that $T$ does not act on $F$ as an isometry. By the singular decomposition it is possible to find orthobases of $F$ and $TF$ such that the matrix of $T\left|_{F}\right.$ with respect to them is $\diag \left\{a,b,c\right\}$, where $a,b,c\ge 0$. Then $ab=bc=ca=1$ by our condition on the two-dimensional restrictions of $T$, and so $a=b=c=1$, which contradicts the assumption that $T$ does not act on $F$ as an isometry.\qed

\begin{remark} \label{dim2}
If $\dim H=2$, the statement of the theorem needs an adjustment, since Step 4 of the proof is not applicable to this case. From steps 2-3 of the proof either $T$ is a congruence, or $\det T= 1$ and $\rho_{g}\left(h\right)=b\|g\|\|h\|\sin\left(\angle\left(g,h\right)\right)$. Hence there is a constant, such that for any non-colinear $g,h$ we have that $\rho_{g}\left(h\right)$ is this constant times the area of the parallelogram formed by $g,h$ (determinant of $2\times 2$ matrix with columns $g,h$). In particular $\rho$ cannot be induced by $\sigma$ as in Proposition \ref{rm1}.
\end{remark}

\section{The non-symmetric case}\label{3}

The proofs of the following three propositions are omitted due to their similarity with the symmetric case. Remarks \ref{dim}, \ref{pr} and \ref{pr2} are also adaptable to this case. For non-zero vectors $g,h$ define the \emph{angle} between them by $\measuredangle \left(g,h\right)=\cos^{-1}\frac{\left<h,g\right>}{\|h\|\|g\|}$. Notice that the difference in the definitions of the angle and the acute angle is the absolute value of the inner product in the latter.

Let $P$ be $\left[0,\pi\right]$ when $\F=\R$, and $\cos^{-1}\left\{\tau\in\C,\left|\tau\right|\le 1\right\}$, for any branch of $\cos^{-1}$, when $\F=\C$.

\begin{proposition}
Let $\rho$ be a non-symmetric Finsler metric on $G\not\ni 0$. The following are equivalent:
\item[(i)] $\rho$ is isometry-invariant;
\item[(ii)] there is a (unique) function $\theta:R\times P\to\R$ such that $\rho_{g}\left(h\right)=\|h\|\theta\left(\|g\|,\measuredangle \left(g,h\right)\right)$, for every $g\in G, h\in H\backslash \left\{0\right\}$;
\item[(iii)] there is a (unique) function $\lambda:R\times\F\times\R\to\R$, such that $\lambda_{r}\left(tp, \pm tq\right)=t\lambda_{r}\left(p,q\right)$, for any $r\in R, p\in \F, q\in\R, t\ge 0$ and $\rho_{g}\left(h\right)=\lambda_{\|g\|}\left(\left<h,g\right>,\sqrt{\|h\|^{2}\|g\|^{2}-\left|\left<h,g\right>\right|^{2}}\right)$, for every $g\in G, h\in H$.
\end{proposition}

\begin{proposition}\label{h}
An isometry-invariant non-symmetric Finsler metric $\rho$ on $H\backslash\left\{0\right\}$ is invariant with respect to a homotethy with a coefficient $\alpha$ if and only if there is a (unique)  function $\vartheta:R\times P\to\R$, such that $\vartheta\left(\exp\left(\cdot\right),\cdot\right)$ is periodic in the first variable with period $\ln\alpha$ and  $\rho_{g}\left(h\right)=\frac{\|h\|}{\|g\|}\vartheta\left(\|g\|,\measuredangle \left(g,h\right)\right)$, for every $g\in G, h\in H\backslash \left\{0\right\}$.
\end{proposition}

\begin{proposition} Let $\rho$ be a non-symmetric Finsler metric on $H\backslash\left\{0\right\}$. The following are equivalent:
\item[(i)] $\rho$ is invariant with respect to all congruencies;
\item[(ii)] there is a (unique) function $\vartheta: P\to\R$ such that $\rho_{g}\left(h\right)=\frac{\|h\|}{\|g\|}\vartheta\left(\measuredangle \left(g,h\right)\right)$, for every $g, h\in H\backslash \left\{0\right\}$;
\item[(iii)] there is a (unique) positive-homogenous function $\lambda:\F\times\R\to\R$, which is even in the second variable such that  $\rho_{g}\left(h\right)=\frac{1}{\|g\|^{2}}\lambda\left(\left<h,g\right>,\sqrt{\|h\|^{2}\|g\|^{2}-\left|\left<h,g\right>\right|^{2}}\right)$, for every $g, h\in H\backslash \left\{0\right\}$.

Moreover, if we additionally assume that $\rho$ is continuous in the first variable, then the above conditions are equivalent to
\item[(iv)] $\rho$ is invariant with respect to all isometries and two homotheties such that the logarithms of their coefficients are not commensurable.
\end{proposition}

\begin{theorem}\label{nain}
If $\dim H>2$, then any symmetry of any non-zero non-symmetric isometry-invariant Finsler metric is a congruence.
\end{theorem}

While if $\F=\R$ the proof of Theorem \ref{main} can be adapted to the non-symmetric case (we would just have to consider all angles, not just acute ones), for the complex case we face an insurmountable obstacle of non-symmetry of the inner product. Therefore we give a completely different proof.

\begin{proof}
Assume that $T$ is a symmetry of a non-zero non-symmetric isometry-invariant Finsler metric $\rho$, which is not a congruence. Replicating Step 0 of the proof of Theorem \ref{main}, we can show that $T$ has to be injective. Let $\mathcal{G}$ be the monoid of all symmetries of $\rho$. We know, that it contains all isometries and $T$. Our goal is to show that $\mathcal{G}$ acts bitransitively on $H$, in the sense that it can map any pair of non-collinear vectors into any other pair. It is clear that a function, invariant with respect to a bitransitive action has to be a constant on the set of all non-collinear pairs, which is not compatible with the assumption that $\rho$ is a non-zero Finsler metric. Let $f_{1},g_{1}$ and $f_{2},g_{2}$ be two non-collinear pairs. Fix a subspace $E$ of $H$, such that:
\begin{itemize}
\item $f_{1},g_{1},f_{2},g_{2}\in E$;
\item $\8>\dim E>2$;
\item $T$ does not act like a congruence from $E$ into $TE$.
\end{itemize}
Let $\mathcal{G}\left|_{E}\right.=\mathcal{G}\bigcap GL\left(E\right)$, i.e. the group of the \textbf{invertible} restrictions on $E$ of the elements of $\mathcal{G}$ that fix $E$. We will use an apparently well-known fact that any linear group that contains the group of rotations of a finite-dimensional inner product space is either contained in the group of congruencies, or contains the special linear group. Since $\mathcal{G}$ contains all isometries of $H$, it follows that $\mathcal{G}\left|_{E}\right.$ contains all rotations of $E$. Let $S$ be an isometry that sends $TE$ back to $E$. Then $ST\left|_{E}\right.\in \mathcal{G}\left|_{E}\right.$, since\\ $ST\left(E\right)\subset E$, and both $T,S\in\mathcal{G}$ are an injections. However $ST\left|_{E}\right.$ is not a congruence on $E$ by the construction of $E$. Hence, $\mathcal{G}\left|_{E}\right.$ is not contained in the group of congruencies, and so $SL\left(E\right)\subset \mathcal{G}\left|_{E}\right.$. It is easy to see that any $n$-tuple of linearly independent vectors in $E$ can be mapped into any other $n$-tuple by a transformation from $SL\left(E\right)$, where $n<\dim E$. Hence, there is an element of $\mathcal{G}$, whose restriction maps $f_{1},g_{1}$ into $f_{2},g_{2}$. Since the latter pairs were arbitrary, the bitransitivity of the action of $\mathcal{G}$ follows.
\end{proof}

\begin{remark}
If $\dim H=2$, $SL\left(H\right)$ does not act bitransitively anymore. Instead, it can map any pair of vectors into any other with the same determinant. Since $\mathcal{G}$ also contains all matrix of the form  $\diag\left\{a,1\right\}$, where $\left|a\right|=1$, we end up with the same conclusion as in the Remark \ref{dim2}.
\end{remark}

\section{Acknowledgements}

The author wants to thank: his supervisor Nina Zorboska for an indispensable help on every stage of the preparation of this paper; Yves de Cornulier, who has drawn the author's attention to the fact that the proof of Theorem \ref{nain} is based upon and the service \href{http://mathoverflow.net/}{MathOverflow} which made it possible.

\end{document}